\documentclass[a4paper,11pt]{article}

\usepackage[T2A]{fontenc}
\usepackage[cp1251]{inputenc}
\usepackage{amsfonts,amsmath,amsthm,amscd,amssymb,latexsym,amsbsy,pb-diagram,cite,caption}
\usepackage[mathscr]{eucal}

\usepackage{ifpdf}
\ifpdf
\usepackage[pdftex]{graphicx}
\usepackage[pdftex,dvipsnames,usenames]{color}
\else
\usepackage{graphicx}
\usepackage[dvips,dvipsnames,usenames]{color}
\fi

\newtheorem{definition}{Definition}[section]
\newtheorem{theorem}[definition]{Theorem}
\newtheorem{lemma}[definition]{Lemma}
\newtheorem{proposition}[definition]{Proposition}
\newtheorem{corollary}[definition]{Corollary}

\theoremstyle{definition}
\newtheorem{remark}[definition]{Remark}
\newtheorem{example}[definition]{Example}

\numberwithin{equation}{section}

\topmargin = - 0.5 cm
\begin{document}

\begin{center}
{\Large\bf Subdominant pseudoultrametric on graphs}
\end{center}

\begin{center}
{\bf O. Dovgoshey and  E. Petrov}
\end{center}

\begin{abstract}
Let $(G,w)$ be a weighted graph. The  necessary and sufficient conditions under
 which a weight $w : E(G)\rightarrow \mathbb{R}^+$ can be extended to a pseudoultrametric on $V(G)$
  are found. A criterion of the uniqueness of this extension is also obtained. It is
   proved that $G$  is  complete k-partite with $k \geq 2$ if and only if, for every
   pseudoultrametrizable weight $w$, there exists the smallest pseudoultrametric agreed
    with $w$. We characterize the structure of graphs for which the subdominant  pseudoultrametric
     is an ultrametric for every strictly positive pseudoultrametrizable weight.
\end{abstract}


\bigskip
{\bf 2010 AMS Classification:} 05C10, 05C12, 54E35.

\section{Introduction}
Throughout this paper, a \textit{graph} is a pair $(V,E)$ consisting of nonempty set $V$ and (probably empty) set $E$  elements of
which are unordered pairs of different points from $V$.
For the graph $G=(V,E)$, the set $V=V(G)$ and $E=E(G)$ are called
\textit{the set of vertices} and, respectively, \textit{the set of edges}.
Generally we shall follow terminology adopted in \cite{BM}.
Let us give some definitions.
If $V(H)\subseteq V(G)$ and $E(H)\subseteq E(G)$, then graph $H$ is a {\it subgraph} of graph $G$, $H\subseteq G$.
Recall that $G$ is called \textit{complete} if every two different vertices $u$, $v$ are \textit{adjacent}, $\{u,v\} \in E(G)$.
Graph $G$ is {\it finite} if $|V(G)|<\infty$.
If $E(G)=\varnothing$, then $G$ is  an \textit{empty graph}.
A finite nonempty graph $P\subseteq G$ is a \textit{path} (in $G$), if we can enumerate without repetition vertices from $P$  into a
 sequence $(v_1,v_2,...,v_n)$ such that
$$
(\{v_i,v_j\}\in E(P))\Leftrightarrow (|i-j|=1).
$$
We shall identify the path $P$ with the  sequence $(v_1,v_2,...,v_n)$ and shall say that $P$ connects $v_1$ and $v_n$.
A finite graph $C$ is a \textit{cycle} if $|V(C)|\geq 3$ and there exists an enumeration $(v_1,v_2,...,v_n)$
 of his vertices such that
\begin{equation*}
(\{v_i,v_j\}\in E(C))\Leftrightarrow (|i-j|=1\quad \mbox{or}\quad |i-j|=n-1).
\end{equation*}
Some two vertices in graph are \textit{connected} if there exists a path connecting them.
A graph is {\it connected} if every two his vertices are connected.
A graph $G=(V,E)$ together with function $w:E\rightarrow \mathbb{R}^+=[0,+\infty)$ is called a \textit{weighted} graph, and  $w$ is
called a \textit{weight} or a \textit{weighting function}.
The weighted graphs we shall denote by $(G,w)$.

Recall now some necessary definitions from the theory of metric spaces.
An \textit{ultrametric} on a set $X$ is a function $d:X\times X\rightarrow \mathbb{R}^+$ such that for all $x,y,z \in X$:
\begin{itemize}
\item [(i)] $d(x,y)=d(y,x)$,
\item [(ii)] $(d(x,y)=0)\Leftrightarrow (x=y)$,
\item [(iii)] $d(x,y)\leq \max \{d(x,z),d(z,y)\}$.
\end{itemize}
If (ii) is replaced by the weaker condition (ii') $d(x,x)=0$, then $d$ is  a {\it pseudoultrametric}.
Inequality (iii)  is often called the {\it strong triangle inequality}.
A function $d:X\times X\rightarrow \mathbb{R}^+$ satisfying the  ordinary triangle inequality and having properties (i)-(ii'), is called a
{\it pseudometric}.

If $(G,w)$ is a weighted graph and
\begin{equation}\label{eq1.2}
2\max\limits_{e\in E(C)} w(e) \leq \sum\limits_{e \in E(C)} w(e)
\end{equation}
for every cycle $C\subseteq G$, then there exists a pseudometric $d:V(G)\times V(G)\rightarrow \mathbb{R}^+$ such that
\begin{equation}\label{eq1.3}
w(\{x,y\})=d(x,y)
\end{equation}
for every $\{x,y\} \in E(G)$.
As an example of  such pseudometric, for connected $G$, we can take the  well known ``shortest path metric''.
This result was proved in~\cite{DMV} and the next question was formulated.
Under what conditions on $w$ there exists an ultrametric (pseudoultrametric) $d$ extending the weight $w$, in the sense
that~(\ref{eq1.3}) holds for all edges $\{x,y\}$ of  $G$?

Theorem~\ref{th2.4} below gives us a complete answer on this question.
The necessary and sufficient conditions of uniqueness of such extension are found in Theorem~\ref{th5.6}.
Moreover, for connected $G$, we find the ``greatest'' pseudoultrametric $d$, extending $w$, and we show, that this
pseudoultrametric is subdominant for the ``shortest path metric'' (see Theorem~\ref{th2.8} and corollary~\ref{cor2.7}).
The necessary and sufficient conditions under  which the subdominant pseudoultrametric is a metric are found in Theorem~\ref{th3.3}.
Using this theorem in Corollary~\ref{cor3.2} we find  the structural characteristic of  graphs $G$, for which there exists
$w:E(G)\rightarrow \mathbb{R}^+$ such that:
\begin{itemize}
\item [(i)] $w(e)>0$ for all $e \in E(G)$;
\item [(ii)] The set of pseudoultrametrics, extending $w$, is not empty, but does not contain any ultrametric.
\end{itemize}
Moreover, we give some results, showing that the subdominant pseudoultrametric and the shortest path metric ``behave similarly''.

\section{Subdominant pseudoultrametric}

In the next lemma and further we identify a pseudoultrametric space $(X,d)$ with the complete weighted graph $(G, w_d)$
 having $V(G)=X$ and satisfying the equality
\begin{equation}\label{eq2.4}
	w_d(\{x,y\})=d(x,y)
\end{equation}
for every pair of different points $x,y \in X$.
	
\begin{lemma}\label{lem2.2}
Let $(X,d)$ be a pseudoultrametric space. Then for every cycle $C\subseteq G(X)$ there exist at least two distinct edges
$e_1$, $e_2$ such that
\begin{equation}\label{eq2.5}
	w_d(e_1)=w_d(e_2)=\max\limits_{e\in E(C)}w_d(e).
\end{equation}
\end{lemma}
\begin{proof}
Let us denote by $q(C)$ the number of edges of a cycle $C$.
If $q(C)=3$, then ~(\ref{eq2.5}) follows from the strong triangle inequality.
Suppose that~(\ref{eq2.5}) holds when $q(C)\leq n$, but there exists a cycle $C$ with $q(C)=n+1$, having exactly one edge $e_1=\{x,y\}$
such that
	$$
	w_d(e_1)=\max\limits_{e\in C}w_d(e).
	$$
Let $z$ be a vertex of cycle $C$ adjacent to   $y$ and distinct  from $x$.
By the uniqueness of the edge of maximal weight we have
	$$
	d(y,z)<d(x,y).
	$$
This inequality and the strong triangle inequality imply  $d(x,z)=d(x,y)$.
Let $C_1$ be a cycle for which
	$$
	V(C_1)=V(C)\backslash\{y\}\quad \mbox{и} \quad E(C_1)=(E(C)\backslash \{\{x,y\},\{y,z\}\})\cup \{\{x,z\}\}.
	$$
Then $q(C_1)=n$ and $\{x,z\}$ is the unique edge of maximal weight, which contradicts the induction hypothesis.
	\end{proof}
\begin{remark}\label{rem2.2}
Probably, this lemma is known.
In any case, the presence in the graph of two edges of  maximum length is a commonly meeting phenomenon under the work with the ultrametrics
 and their generalizations.
For example, the  so called 2-ultrametric spaces are characterized by the fact that every their four-point subspace has at least
two edges with the length equals to the diameter of the subspace (see~\cite{JS}).
\end{remark}

We turn now to the definition of the subdominant pseudoultrametric.

On the set $\mathfrak{F}$ of  the pseudometrics defined on $X$ we introduce the  partial order $\preceq$ as
\begin{equation}\label{eq2.02}
(d_1\preceq d_2)  \Leftrightarrow (\forall  x,y \in X:d_1(x,y)\leq d_2(x,y)).
\end{equation}

In~\cite{BMM}, for given metric space $(X,d)$, the   subdominant ultrametric is defined as the  greatest element of the
 poset $(\mathfrak{F}_d,\preceq)$, where $\mathfrak{F}_d \subseteq \mathfrak{F}$ is the set of the ultrametrics $\delta$ such that
$$
\delta(x,y)\leq d(x,y)
$$
for $x,y\in X$.
We generalize this definition to the  weighted graphs.

\begin{definition}\label{def2.3*}
Let $(G,w)$ be a nonempty weighted graph  and $\mathfrak{F}_{w,u}$ be the family of the pseudoultrametrics $\rho$ such that
$$
\rho(u,v)\leq w(\{u,v\})
$$
for every edge $\{u,v\}\in E(G)$.
If the poset $(\mathfrak{F}_{w,u},\preceq)$ contains the greatest element, then we call this element the subdominant pseudoultrametric for
 $w$.
\end{definition}

Note that $\mathfrak{F}_{w,u}\neq \varnothing$ because the zero pseudoultrametric
$$
\rho(u,v)=0, \forall u,v \in V(G)
$$
belongs to $\mathfrak{F}_{w,u}$.

We turn now to the construction of subdominant pseudoultrametrics.

Let $u$, $v$ be two distinct vertices of a connected weighted graph $(G,w)$.
Denote by $\mathfrak{P}_{u,v}$ the  set of the paths connecting $u$ and $v$.
Define the function $\rho_w$ on the Cartesian square $V(G)\times V(G)$ by the rule
\begin{equation}\label{eq2.1}
\rho_{w}(x,y):=\left\{
\begin{array}{l}
0 \mbox{ if } x=y \\\
\inf\limits_{P\in \mathfrak{P}_{x,y}}(\max\limits_{e\in P}w(e)) \mbox{ if } x \neq y.\\
\end{array}
\right.
\end{equation}
\begin{theorem}\label{st2.1}
The function $\rho_{w}$ is the subdominant pseudoultrametric for every nonempty connected weighted graph $(G,w)$.
\end{theorem}
\begin{proof}
Let us verify the strong triangle inequality
\begin{equation}\label{eq2.2}
\rho_{w}(u,v)\leq \max\{\rho_{w}(u,p),\rho_{w}(p,v)\}
\end{equation}
for different vertices $u,v,p \in V(G)$.
Let $\varepsilon$ be an  arbitrary positive number.
There exist some paths $P_1 \in \mathfrak{P}_{u,p}$ and $P_2 \in\mathfrak{P}_{p,v}$ such that
\begin{equation}\label{eq2.3}
\rho_{w}(u,p)+\varepsilon \geq \max\limits_{e\in P_1}w(e)\quad \mbox{и} \quad \rho_{w}(p,v)+\varepsilon \geq \max\limits_{e\in P_2}w(e).
\end{equation}
The subgraph $G_1 $ of $G$ with $V(G_1)=V(P_1)\cup V(P_2)$ and $E(G_1)=E(P_1)\cup E(P_2)$ is connected.
Let $P_3$ be a path in $G_1$, connecting $u$ and $v$.
Then using~(\ref{eq2.3}) we find
$$
\rho_{w}(u,v) \leq \max\limits_{e\in P_3}w(e) \leq (\max\limits_{e\in P_1}w(e)) \vee (\max\limits_{e\in P_2}w(e))\leq \max \{\rho_w(u,p)+\varepsilon,\rho_w(p,v)+\varepsilon\}.
$$
Hence,  letting $\varepsilon$ to zero, we obtain~(\ref{eq2.2}).

It remains to verify that $\rho_w$ is subdominant.
Suppose there exist $\rho \in \mathfrak{F}_{w,u}$ and $v_1, v_2 \in V(G)$ such that
\begin{equation}\label{eq2.9}
\rho(v_1,v_2)>\rho_w(v_1,v_2).
\end{equation}
This inequality and~(\ref{eq2.1}) imply the  existence of a path  $P\in \mathfrak{P}_{v_1,v_2}$ for which
$$
\rho(v_1,v_2)>\max\limits_{e\in P} w(e).
$$
Note that $\rho(u,v)\leq w(\{u,v\})$ for every $\{u,v\}\in E(G)$.
Consequently, the path $P$ does not contain $\{v_1,v_2\}$.
Consider, in the pseudoultrametric space $(V(P),\rho)$,  a cycle  $C$ with
$$
V(C)=V(P), \quad E(C)=E(P)\cup\{\{v_1,v_2\}\}.
$$
Then $\{v_1,v_2\}$ is the unique edge of  $C$ on which $\max\{\rho(x,y):\{x,y\}\in E(C)\}$ is achieved, contrary
to  Lemma~\ref{lem2.2}.
Thus, for every  $v_1,v_2 \in V(G)$ and  $\rho \in \mathfrak{F}_{w,u}$ the  inequality $\rho(v_1,v_2)\leq\rho_w(v_1,v_2)$ holds,
i.e. $\rho_w$ is the greatest element of  $(\mathfrak{F}_{w,u},\preceq)$.
\end{proof}

\begin{remark}\label{rem2.55}
If $G$ is a finite graph and a weight $w$ is defined by some metric as in~(\ref{eq1.3}), then the subdominant pseudoultrametric
$\rho_w$ is an ultrametric.
For the complete $G$ this classic case was considered in~\cite{JJS}.
An efficient procedure  for the evaluation of the subdominant ultrametric on the finite metric spaces  can be found in ~\cite{RAD}
and~\cite{RTV}.
\end{remark}

We now turn to the study of connections between the  shortest-path pseudometric and the subdominant pseudoultrametric.
Remind that the shortest-path pseudometric is a pseudometric defined on $V(G)$ as
\begin{equation}\label{eq2.10*}
d_{w}(x,y)=\left\{
\begin{array}{l}
0, \mbox{ if } x=y \\\
\inf\limits_{P\in \mathfrak{P}_{x,y}}\sum\limits_{e\in E(P)}w(e), \mbox{ if } x \neq y.\\
\end{array}
\right.
\end{equation}

Similarly to Definition~\ref{def2.3*} we introduce
\begin{definition}\label{def2.5}
Let  $(G,w)$ be  a nonempty weighted graph and $\mathfrak{F}_{w,m}$ be the set of the pseudometrics $d$ such that
\begin{equation}\label{eq2.11*}
d(u,v)\leq w(\{u,v\})
\end{equation}
for all edges $\{u,v\}\in E(G)$.
The greatest element of the poset $(\mathfrak{F}_{w,m}\preceq)$ is called, if it exists, the subdominant pseudometric (for the weight $w$).
\end{definition}

\begin{proposition}\label{prop2.6}
Let $(G,w)$ be a nonempty connected weighted graph.
Then $d_w$ is the subdominant pseudometric for the weight $w$.
\end{proposition}
\begin{proof}
The inequality $d_w(u,v)\leqslant w(\{u,v\})$ holds for every $\{u,v\}\in E(G)$.
This follows from ~(\ref{eq2.10*}) and the fact that the two-term sequence $u,v$ is a path belonging to $\mathfrak{P}_{u,v}$.
Consequently, $d_w \in \mathfrak{F}_{w,m}$.
Suppose that there are $d \in \mathfrak{F}_{w,m}$ and $u,p \in V(G)$ for which
$d(u,p)>d_w(u,p)$.
Then there exists a path $(u=x_1,...,x_n=p) \in \mathfrak{P}_{u,p}$ such that
\begin{equation}\label{eq2.12*}
d(u,p)>\sum\limits_{i=1}^{n-1}w(\{x_i,x_{i+1}\}).
\end{equation}
Since $d\in \mathfrak{F}_{w,m}$, the inequality $w(\{x_i,x_{i+1}\})\geq d(x_i,x_{i+1})$ holds for $i=1,...,n-1$.
From here and~(\ref{eq2.12*}) we find
$$
d(x_1,x_n)\geq \sum\limits_{i=1}^{n-1}d(x_i,x_{i+1}),
$$
contrary to the triangle inequality.
\end{proof}

\begin{corollary}\label{cor2.7}
Let $(G,w)$ be a nonempty connected weighted graph.
Then the pseudoultrametric  $\rho_w$ constructed by  rule~(\ref{eq2.1}) is the subdominant pseudoultrametric for $d_w$, i.e.,
$\rho_w \preceq d_w \mbox{ and } \rho \preceq \rho_w$ for every pseudoultrametric $\rho$ satisfying $\rho \preceq d_w$.
\end{corollary}
\begin{proof}
Denote by $\rho_w^*$ the subdominant pseudoultrametric for $d_w$.
From~(\ref{eq2.4}), ~(\ref{eq2.10*}) and~(\ref{eq2.1})  it follows that $\rho_w\preceq d_w$, consequently $\rho_w\preceq\rho_w^*$.
The inverse relation $\rho_w^*\preceq \rho_w$ follows from the fact that every pseudoultrametric $\rho$ satisfying
 $\rho \preceq d_w$ belongs to the set  $\mathfrak{F}_{w,u}$ (see Definition~\ref{def2.3*}).
Thus, $\rho_w^*= \rho_w$.
\end{proof}
\begin{remark}\label{rem2.8}
The question when  $d_w$ and $\rho_w$ are  metrics  is of interest in its own right.
We return to this in Section 4.
Note that the problem of finding a criterion of existence of the subdominant \textbf{ultrametric}, for a given metric,  was posed
in \cite{M}.
\end{remark}

\section{Pseudoultrametrization of weighted graphs}

\begin{definition}\label{def3.1}
Let $(G,w)$ be a weighted graph and let $d:V(G)\times V(G)\rightarrow \mathbb{R}^+$ be an ultrametric (pseudoultrametric)
We shall say that $d$ extends $w$ if~(\ref{eq1.3}) holds for every  $\{x,y\} \in E(G)$.
\end{definition}

\begin{definition}\label{def3.2}
The weight $w$ is \textit{ultrametrizable} (\textit{pseudoultrametrizable}) if there exists an ultrametric (pseudoultrametric)
extending $w$.
\end{definition}

The following theorem gives us a criterion of pseudoultrametrizability  for a given weight $w$.

\begin{theorem}\label{th2.4}
Let $(G,w)$ be a nonempty weighted graph. The weight $w$ is  pseudoultrametrizable if and only if,  for each cycle $C\subseteq G$,
there exist at least two distinct edges $e_1, e_2\in E(C)$ such that
\begin{equation}\label{eq2.6}
w(e_1)=w(e_2)=\max\limits_{e\in E(C)} w(e).
\end{equation}
If $G$ is a connected graph and $w$ is a pseudoultrametrizable weight, then the subdominant pseudoultrametric extends $w$.
\end{theorem}
\begin{proof}
Lemma~\ref{lem2.2} implies that~(\ref{eq2.6}) holds for each cycle  $C\subseteq G$ if $w$ is a pseudoultrametrizable weight.

Conversely, suppose that for each cycle  $C \subseteq G$ there exist $e_1, e_2 \in E(C)$ such that~(\ref{eq2.6}) holds and prove
the pseudoultrametrizability of  $w$.
Consider first the case, where $G$ is connected.
By Theorem~\ref{st2.1} it is sufficient to prove, for every edge $\{u,v\}\in E(G)$, the following equality
\begin{equation}\label{eq2.7}
\rho_w(u,v)=w(\{u,v\}),
\end{equation}
where $\rho_w$ is the subdominant pseudoultrametric for  $w$.
By Definition~\ref{def2.3*}, we have
$$
\rho_w(u,v)\leq w(\{u,v\}).
$$
If this inequality is strict, then there exists a path $P\in \mathfrak{P}_{u,v}$ for which
$$
\max\limits_{e\in P}w(e)<w(\{u,v\}).
$$
The last inequality implies that $\{u,v\}\notin E(P)$.
Since $\{u,v\}\in E(G)$, there exists a cycle $C$ with
$$
V(C)=V(P), \quad E(C)=E(P)\cup \{\{u,v\}\}.
$$
If $e_i$ is an edge of $C$ different form $\{u,v\}$, then $e_i\in E(P)$, so that
$$
w(e_i)\leq\max\limits_{e\in P}w(e)<w(\{u,v\})=\max\limits_{e\in C}w(e),
$$
contrary to~(\ref{eq2.6}).
The pseudoultrametrizability of $w$ is proved for the connected $G$.
\begin{figure}[ht]
\centering
\includegraphics[width=0.8\linewidth]{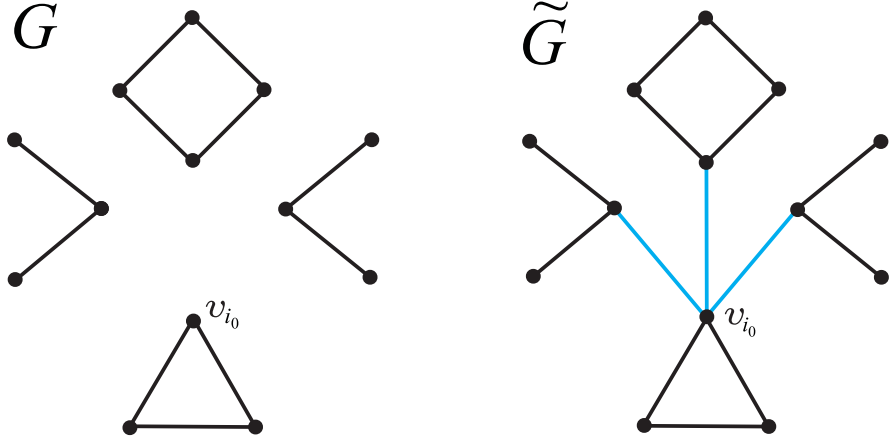}
\caption{The transition from a disconnected graph $G$ to a connected $\tilde{G}$.}
\label{ris1}
\end{figure}

Let $G$ now be disconnected.
Consider the set $\{G_i:i\in \mathfrak{I}\}$ of connected components of $G$, where $\mathfrak{I}$ is an indexing set.
For each $i \in \mathfrak{I}$ choose $v_i \in V(G_i)$ and fix an index $i_0 \in \mathfrak{I}$.
Consider a new graph $\tilde{G}$ with
$$
V(\tilde{G})=V(G), \quad E(\tilde{G})=E(G)\cup\{\{v_i,v_{i_0}\}:i \in \mathfrak{I}\backslash \{i_0\}\}.
$$
Let us extend $w$ to a function $\tilde{w}:E(\tilde{G})\rightarrow \mathbb{R}^+$ as
\begin{equation}\label{eq2.8}
\tilde{w}(e)=\left\{
\begin{array}{l}
w(e) \mbox{ if } e \in E(G) \\\
c_i \mbox{ if } e=\{v_i,v_{i_0}\}, i \in \mathfrak{I}\backslash\{i_0\}\\
\end{array}
\right.
\end{equation}
where $c_i$ are arbitrary non-negative constants.
Since $\tilde{G}$ is a connected graph, to prove the pseudoultrametrizability of $G$ it is sufficient to 
establish ~(\ref{eq2.6}) with $w=\tilde{w}$ for every cycle $C \subseteq \tilde{G}$.
Observe that each cycle $C\subseteq \tilde{G}$ is a cycle in $G$.
It is easily seen by drawing an appropriate picture (see Fig.~\ref{ris1}).
The formal proof is as follows.
Let $C \subseteq \tilde{G}$ but $C \nsubseteq G$.
Then there exists an edge in $C$ of the form $\{v_{i_0},v_{i_1}\}$, $i_1 \in \mathfrak{I}\backslash \{i_0\}$ and there exists a
unique edge incident to $v_{i_0}$ in $C$ and different from $\{v_{i_0},v_{i_1}\}$.
The definition of  $\tilde{G}$ implies that this edge has the form $\{v_{i_0},v_{i_2}\}$.
Removing, from the cycle $C$, the vertex $v_{i_0}$ we get the path $P$,
$$
V(P)=V(C)\backslash \{v_{i_0}\},\quad E(P)=E(C)\backslash \{\{v_{i_0},v_{i_1}\},\{v_{i_0},v_{i_2}\}\},
$$
connecting $v_{i_1}$ and $v_{i_2}$.
Since $E(P)\subseteq E(G)$, $v_{i_1}$ and $v_{i_2}$ lie in the same connected component, which contradicts to their definition.
Consequently if $C\subseteq \tilde{G}$, then $C \subseteq G$, so that~(\ref{eq2.6}) follows.
\end{proof}
\begin{remark}\label{rem2.5}
Condition~(\ref{eq2.6}) is  equivalent to the strong triangle inequality if the cycle $C$ contains exactly three vertices.
Note also that
$$
\sum\limits_{e \in E(C)}w(e)\geq w(e_1)+w(e_2)=2\max\limits_{e \in E(C)}w(e)
$$
if~(\ref{eq2.6}) holds.
Thus, the condition of the pseudoultrametrizability of the weight implies the condition of its pseudometrizability, as expected.
\end{remark}

\begin{remark}\label{rem3.5}
Directly from  Theorem~\ref{th2.4} it  follows  that the pseudoultrametrizability of the weight is a local property, i.e., if
 for every \textbf{finite} subgraph  $H$ of a weighted graph $(G,w)$ the restriction of $w$ on $E(H)$ is  pseudoultrametrizable, 
 then the weight $w$ is pseudoultrametrizable too.
\end{remark}

Recall that a graph which does not contain any cycles is called a \textit{forest} and  a \textit{tree} is a
 connected forest.

\begin{corollary}\label{cor2.6}
Let $G$ be a graph with $V(G)\neq \varnothing$.
$G$ is a forest if and only if every weight $w:E(G)\rightarrow \mathbb{R}^+$ is pseudoultrametrizable.
\end{corollary}
\bigskip
\noindent For the proof it suffices
to note that the existence of a cycle $C\subseteq G$ implies the existence of a weight $w:E(G)\rightarrow \mathbb{R}^+$ for which
 condition ~(\ref{eq2.6}) does not hold.

Let $(G,w)$ be a weighted graph with a pseudoultrametrizable weight $w$.
Denote by $\mathfrak{U}_w$ the set of all pseudoultrametrics on  $V(G)$ extending $w$.

\begin{theorem}\label{th2.8}
If $G$ is connected, then the subdominant pseudoultrametric $\rho_w$ is the greatest element of the poset
$(\mathfrak{U}_w,\preceq)$.
Conversely, if  $(\mathfrak{U}_w,\preceq)$ has the greatest element, then  $G$ is connected.
\end{theorem}
\begin{proof}
Let $\mathfrak{F}_{w,u}$ be the set from Definition~\ref{def2.3*}.
Suppose $G$ is a connected graph.
Then $\mathfrak{F}_{w,u} \supseteq \mathfrak{U}_w$ and the subdominant pseudoultrametric $\rho_w$ belongs to  $\mathfrak{F}_{w,u}$.
By the definition of the subdominant pseudoultrametric we have $\rho\leqslant\rho_w$ for every $\rho \in \mathfrak{U}_w$.
To prove that $\rho_w$ is the greatest element of $(\mathfrak{U}_w,\preceq)$ it is sufficient to verify the relation
$\rho_w \in \mathfrak{U}_w$ that has already been established in Theorem~\ref{th2.4}.

Suppose $G$ is not connected.
Fix some points $v_{i_0}$ and $v_{i_1}$ belonging to distinct connected components.
Let us consider the weighted graph $(\tilde{G},\tilde{w})$ as it was done in the proof of Theorem~\ref{th2.4}.
It is clear that $\mathfrak{U}_w\supseteq \mathfrak{U}_{\tilde{w}}$.
The last inqlusion and the arbitrariness of constants $c_i$ in~(\ref{eq2.8}) imply  the equality
$$
\sup\limits_{\rho \in \mathfrak{U}_w}\rho(v_{i_0},v_{i_1})=+\infty.
$$
Consequently, the poset $(\mathfrak{U}_w,\preceq)$ does not contain the greatest element for the disconnected graphs.
\end{proof}

Using the last theorem we can easily obtain the converse assertion to Corollary~\ref{cor2.7}.

\begin{corollary}\label{cor3.7*}
Let $(G,w)$ be a nonempty weighted graph.
If, for $w$, there exists the subdominant pseudoultrametric, then $G$ is connected.
\end{corollary}
\begin{remark}\label{rem3.10}
In corollaries~\ref{cor2.7} and~\ref{cor3.7*}  we do not require the pseudoultrametrizability of $w$.
\end{remark}

In order to trace an analogy between  $\rho_w$ and $d_w$, denote by $\mathfrak{M}_w$ the family of the
 pseudometrics extending $w$.

As it was shown in~\cite{DMV} for connected $G$, the shortest-path pseudometric $d_w$ belongs to  $\mathfrak{M}_w$ for every
pseudometrizable weight $w$.
If we introduce a partial order $\preceq$ in $\mathfrak{M}_w$ as on the subset of
 $(\mathfrak{F},\preceq)$ (see~(\ref{eq2.02})), then the following analog of Theorem~\ref{th2.8} holds.

\begin{theorem}[\cite{DMV}]\label{th2.9}
Let $(G,w)$ be a nonempty weighted graph with the pseudometrizable weight $w$.
If $G$ is connected, then $d_w$ is the greatest element of the poset $(\mathfrak{M}_w,\preceq)$.
Conversely, if $(\mathfrak{M}_w,\preceq)$ has a greatest element, then the graph $G$ is connected.
\end{theorem}

Theorems~\ref{th2.8} and \ref{th2.9} imply
\begin{corollary}\label{cor2.10}
Let $G$ be a nonempty graph. The following statements are equivalent:
\begin{itemize}
\item [(i)] $G$ is a connected graph;
\item [(ii)] The poset $(\mathfrak{M}_w,\preceq)$ has the greatest element for every pseudometrizable weight
$w:E(G)\rightarrow \mathbb{R}^+$;
\item [(iii)] The poset $(\mathfrak{U}_w,\preceq)$ has the greatest element for every pseudoultrametrizable weight
$w:E(G)\rightarrow \mathbb{R}^+$.
\end{itemize}

\end{corollary}

Using Corollary~\ref{cor2.6}, Theorem~\ref{th2.8}, and  the corresponding results from~\cite{DMV} we get
\begin{corollary}\label{cor2.11}
Let $G$ be a nonempty graph. The following statements are equivalent:
\begin{itemize}
\item [(i)] $G$ is a tree;
\item [(ii)] Every weight $w:E(G)\rightarrow \mathbb{R}^+$ is pseudometrizable and the poset $(\mathfrak{M}_w,\preceq)$ contains the
greatest element;
\item [(iii)] Every weight  $w:E(G)\rightarrow \mathbb{R}^+$ is pseudoultrametrizable and the poset $(\mathfrak{U}_w,\preceq)$
contains the greatest element.
\end{itemize}

\end{corollary}

Another examples illustrating the analogy between $\rho_w$ and $d_w$ are given in the next section.
\begin{remark}\label{rem3.12}
If the weight $w$ is pseudometrizable but not pseudoultrametrizable, then the following problem arises.
Find the extension of $w$ being as ``ultrametrizable'' as possible.

We can introduce an appropriate measure of ``ultrametrizability'' using the so-called ``betweenness exponent'' being the supremum
$\alpha \geq 1$ for which $d^{\alpha}$ remains to be a metric for a given metric $d$ (see \cite{DM}, \cite{DD}).

We can easily extend the notion of betweenness exponent to the case of weighted graphs.
If $\alpha =1$, then the betweenness exponent gives us the condition of pseudometrizability of the weight $w$
and,  if $\alpha=\infty$, the  condition of pseudoultrametrizability of $w$.
\end{remark}

\section{Ultrametrization of weighted graphs}

In the previous section it was proved that a weight $w:E(G)\rightarrow \mathbb{R}^+$ is pseudoultrametrizable if and only if
 condition~(\ref{eq2.6}) holds for every cycle $C \subseteq G$.
If a pseudoultrametrizable weight $w$ is strictly positive, i.e., for every $e\in E(G)$ we have
$$
w(e)>0,
$$
and $G$ is finite and connected, then it is clear that the subdominant pseudoultrametric $\rho_w$ is an ultrametric.
The following example shows that, for infinite $G$, the strict positivity of $w$ does not guarantee that $\rho_w$ is an ultrametric.
\begin{example}\label{ex3.1}
Let  $(G,w)$ be an infinite weighted graph, depicted in Figure~\ref{ris2}, where
$$
\varepsilon_n=w(\{u, s_n\})=w(\{s_n,t_n\})=w(\{t_n,v\})
$$
are positive real numbers such that
$\lim\limits_{n\rightarrow \infty}\varepsilon_n=0$ and $\varepsilon_n>\varepsilon_{n+1}$ for each $n$.
\begin{figure}[ht] 
\centering
\includegraphics[width=1\linewidth]{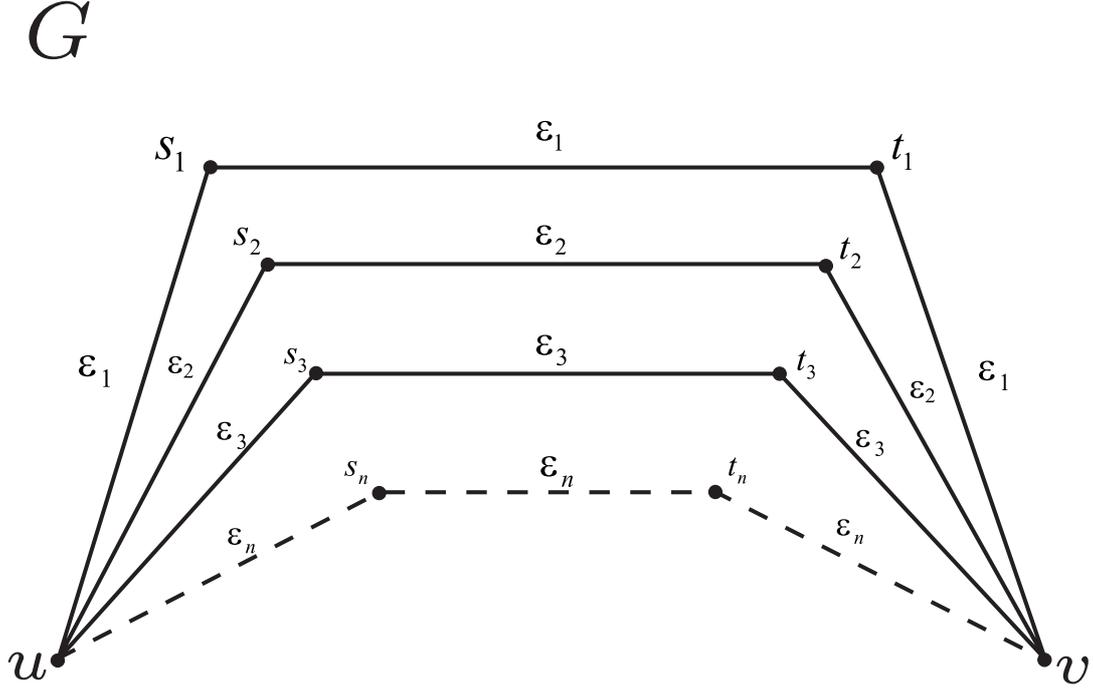}
\caption{The weighted graph with the pseudoultrametrizable weight $w$ such that $\mathfrak{U}_w$ does not contain ultrametrics.}
\label{ris2}
\end{figure}
The length of every cycle $C\subseteq G$ is equal to six and its vertices are $u, s_n, t_n, v, t_m, s_m$, where $m \neq n$.
Such cycle $C$ has the three distinct edges of maximal weight. Consequently,  $w$ is pseudoultrametrizable.
The definition of $\rho_w$ implies that
$$
\rho_w(u,v)=\max\{w(\{u,s_n\}),w(\{s_n,t_n\}),w(\{t_n,v\})\}=\varepsilon_n \vee \varepsilon_m.
$$
Letting $n,m \to \infty$, we obtain $\rho_w(u,v)=0$ .
\end{example}
From Theorem~\ref{th2.8} it follows, for connected $G$, that the set $\mathfrak{U}_w$ contains  ultrametrics if and only if
$\rho_w$ is an ultrametric.
In the present section we shall describe the structure of the graphs $G$ for which $\rho_w$ is an ultrametric for every strictly
 positive
pseudoultrametrizable weight $w$.

Note that  the paper~\cite{Le} contains the  complete characterization of metric spaces $(X,d)$ for which the
subdominant (for $d$) pseudoultrametric is an ultrametric.

We need the following lemma from\cite{DMV}.
\begin{lemma}\label{lemma3.2}
Let $G$ be a connected graph and let $u^*$, $v^*$ be two nonadjacent vertices of $G$.
Let $\tilde{F}=\{F_j\}_{j\in \mathbb{N}}$  be a sequence of paths connecting $u^*$ and $v^*$
and meeting the following condition:
\begin{itemize}
\item [($i_1$)] For every  $e^0 \in E(G)$ there exist $u^0\in e^0$ and $i=i(e^0)$ such that
$u^0 \notin  \bigcup\limits_{k=1}^{\infty} V(F_{i+k})$.
\end{itemize}
Then there exists a subsequence $\{F_{j_k}\}_{k\in \mathbb{N}}$ of the sequence $\tilde{F}$ such that:
\begin{itemize}
\item [($i_2$)] $E(F_{j_l})\cap E(F_{j_k})  = \varnothing$ for $l \neq k$;
\item [($i_3$)] If  $C$ is a cycle in the graph $\bigcup\limits_{k \in \mathbb{N}}F_{j_k}$ and
\begin{equation}\label{eq3.4}
k_0=k_0(C):=\min\{k\in \mathbb{N}:E(C)\cap E(F_{j_k})\neq \varnothing\},
\end{equation}
then $C$ and  $F_{j_{k_0}}$ have at least two common edges.
\end{itemize}
\end{lemma}
\begin{remark}\label{rem4.2}
Here and below by the union  $\bigcup\limits_{i \in \mathcal{I}}G_i$ of subgraphs $G_i$ of the graph $G$ we shall always mean the
 subgraph $\tilde{G}\subseteq G$ for which 
$$
V(\tilde{G})=\bigcup\limits_{i \in \mathcal{I}}V(G_i) \mbox{ and } E(\tilde{G})=\bigcup\limits_{i \in \mathcal{I}}E(G_i).
$$
\end{remark}

The next theorem is the main result of the section.
\begin{theorem}\label{th3.3}
Let $G=(V,E)$ be a nonempty connected graph.
The following two statements are equivalent.
\begin{itemize}
\item [(i)] For every strictly positive pseudoultrametrizable weight $w$ the subdominant pseudoultrametric $\rho_w$ is an ultrametric.
\item [($ii$)] For every pair of distinct points  $u^*$, $v^* \in V(G)$  and for an arbitrary sequence $\tilde{F}$  of
paths $F_j \in \mathfrak{P}_{u^*, v^*}$, $j \in \mathbb{N}$ there exists an edge  $e^0=\{u^0, v^0\}\in E(G)$ such that
$$
u^0, v^0 \in  \bigcup\limits_{k=1}^{\infty} V(F_{i+k}) \mbox{ for every } i>0.
$$
\end{itemize}
\end{theorem}

\begin{proof}
$\textbf{(i)}\Rightarrow \textbf{(ii)}$
Suppose that statement (ii) does not hold.
Then there exist a pair of distinct vertices $u^*$, $v^*$ and a sequence $\tilde{F}$ of paths $F_j\in \mathfrak{P}_{u^*,v^*}$
such that for every edge $e^0 \in E(G)$, $e^0=\{u^0, v^0\}$, there exist $i\in \mathbb{N}$ and at least one vertex incident
to $e^0$, for example $u^0$, for which
\begin{equation*}
u^0 \notin  \bigcup\limits_{k=1}^{\infty} V(F_{i+k}).
\end{equation*}
Let us show that, in this case, statement (i) does not hold.
By Lemma~\ref{lemma3.2}, without loss of generality, it can be assumed that
\begin{equation}\label{eq3.7}
E(F_i)\cap E(F_j)  = \varnothing \mbox{ for } i\neq j,
\end{equation}
and every cycle $C$ from  $\bigcup\limits_{j\in \mathbb{N}}F_j$ has at least two common edges with $F_{k_0}$, where $k_0$ is
defined as in~(\ref{eq3.4}).
Consider the graph
$$
\tilde{G}=\bigcup\limits_{i \in \mathbb{N}}F_{i}.
$$
Let $\{\varepsilon_i\}_{i\in \mathbb{N}}$ be a strictly decreasing sequence of positive real numbers with
$\lim\limits_{i \rightarrow \infty}\varepsilon_i=0$.
Define the weight $w_1:E(\tilde{G})\rightarrow \mathbb{R}^+$ as $w_1(e):=\varepsilon_i$  if $e \in E(F_i)$.
The definition is correct by virtue of~(\ref{eq3.7}).
All  edges of the path $F_i$, $i \in \mathbb{N}$ have the same weight $\varepsilon_i$ and every edge of the path $F_{i+1}$ has a weight strictly
 less than  $\varepsilon_i$.
Every cycle $C\subseteq \tilde{G}$ has at least two common edges $e_1$, $e_2$ with the path $F_{k_0}$,
where $k_0$ is defined by~(\ref{eq3.4}).
Consequently the weights of these edges are maximal,  $w_1(e_1)=w_1(e_2)=\varepsilon_{k_0}=\max\limits_{e\in E(C)}w_1(e)$.
By Theorem~\ref{th2.4}, the weight $w_1$ is pseudoultrametrizable.
Choosing the pseudoultrametrization $\rho_{w_1}$ as in~(\ref{eq2.1}) we get
\begin{equation}\label{eq4*}
\rho_{w_1}(u^*,v^*)=\inf\limits_{i\in \mathbb{N}}(\max\limits_{e\in F_i}w(e))=
\inf\limits_{i\in \mathbb{N}}\varepsilon_i=0.
\end{equation}
Thus $\rho_{w_1}$ is not metric on the set $V(\tilde{G})$.

Using the pseudoultrametric $\rho_{w_1}$ we extend the weighted function $w_1$ to the set of edges of $G$ having
the vertices in $V(\tilde{G})$.
Let us prove that we obtain again a strictly positive weight (for which we keep the same notation $w_1$).
Let $e_0=\{u_0,v_0\}\in E(G)$, $u_0\in V(\tilde{G})$ and $v_0\in V(\tilde{G})$.
By assumption there is at least one end of the edge $e^0$, for example $u^0$, and there exists an index  $i_0$ such that
\begin{equation}\label{eq3.6.1}
u^0 \notin V(F_i)
\end{equation}
if $i>i_0$.
Let $F$ be a path in $\tilde{G}$ connecting $u^0$ and $v^0$ and let $e \in E(F)$ be an edge incident with $u^0$.
From~(\ref{eq3.6.1}) it follows that
$$
e \in \bigcup\limits_{i=1}^{i_0}E(F_i).
$$
Since the sequence $\{\varepsilon_i\}_{i \in \mathbb{N}}$ is decreasing, we get $w(e)\geq  \varepsilon_{i_0}$,
 so that $\max\limits_{e\in F}w(e) \geq \varepsilon_{i_0}>0$.
Thus
$$
w_1(\{u_0,v_0\})=\rho_{w_1}(u^0,v^0)>0.
$$
The next step of our proof is to assign some positive weights for edges $e=\{u,v\}\in E(G)$ having $u \notin V(\tilde{G})$ or
$v \notin V(\tilde{G})$.
Assign for every such edge $e=\{u,v\}$ the weight $w_2(e)=M$, where $M$ is an arbitrary number from $[\varepsilon_1,\infty)$.
Moreover if $e=\{u,v\}$ with $u, v \in V(\tilde{G})$, then set $w_2(e):=\rho_{w_1}(u,v)$.
It is clear that so defined weight $w_2$ is pseudoultrametrizable and $w_2(e)>0$ for every $e\in E(G)$.
Indeed, if $C$ is a cycle in $G$ having all the edges in $\tilde{G}$, then there are two maximum weight edges
in $C$ because  $\rho_{w_1}$ is a pseudoultrametric.
Now let $v \in V(C)$ but $v \notin V(\tilde{G})$. 
Then two edges of the cycle $C$ which are incident with $v$ have the maximal weight $M$.
The pseudoultrametrizability of $w_2$ follows from Theorem~\ref{th2.4}.
Since $\tilde{G}\subseteq G$ and $w_1=w_2|_{E(\tilde{G})}$, from~(\ref{eq2.1}) we obtain
\begin{equation}\label{eq4**}
\rho_{w_2}(u,v)\leq \rho_{w_1}(u,v)
\end{equation}
for $u, v \in V(\tilde{G})$.
Relations~(\ref{eq4*}) and~(\ref{eq4**}) imply $\rho_{w_2}(u^*,v^*)=0$.
Thus we have found the strictly positive pseudoultrametrizable weight $w_2$ for which the subdominant pseudoultrametric
 $\rho_{w_2}$ is not an ultrametric.

$\textbf{(ii)}\Rightarrow \textbf{(i)}$
Let $(G,w)$ be a weighted graph with a pseudoultrametrizable strictly positive weight and let condition (ii) hold.
Let us prove that the pseudoultrametric $\rho_w: V(G)\times V(G)\to \mathbb{R}^+$ is an ultrametric.

Suppose the contrary. Then, for some vertices $u^*$ and $v^*$, $u^*\neq v^*$, we have $\rho_w(u^*,v^*)=0$.
Consequently there exists a sequence $\{F_{k}\}_{k\in \mathbb{N}}$, $F_k\in \mathfrak{P}_{u^*,v^*}$, such that for every $\varepsilon>0$ there exists
$k(\varepsilon) \in \mathbb{N}$ meeting the inequality
\begin{equation}\label{eq3.10*}
\max\limits_{e\in F_{k}}w(e)<\varepsilon
\end{equation}
for  $k \geq k(\varepsilon)$.
By assumption (ii) there exists an edge $e^0=\{u^0,v^0\}\in E(G)$, such that $u^0, v^0 \in \bigcup\limits_{i=1}^{\infty} V(F_{i+k})$
 for all $k>0$.

Let us choose a path $P$ connecting $u^0$ and $v^0$ in the graph $G_{\varepsilon}:=\bigcup\limits_{i=1}^{\infty}F_{k(\varepsilon)+i}$.
The definition of $G_{\varepsilon}$ and~(\ref{eq3.10*}) imply the inequality
$$
w(e)<\varepsilon
$$
for every $e \in P$.
This inequality and the definition of $\rho_w$ give us
$$
\rho_w(u^0,v^0)\leq \max\limits_{e\in P}w(e)<\varepsilon.
$$
Letting $\varepsilon$ to zero, we obtain $\rho_w(u^0,v^0)=0$. Since $w$ is strictly positive and pseudoultrametrizable, by
Theorem~\ref{th2.4} we have
$$
0<w(\{u^0,v^0\})=\rho_w(u^0,v^0).
$$
This contradiction completes the proof.
\end{proof}

This theorem and the corresponding result from~\cite{DMV} imply
\begin{corollary}\label{cor3.4}
Let $G$ be a nonempty connected graph.
The following two statements are equivalent:
\begin{itemize}
\item [(i)]  The subdominant pseudoultrametric $\rho_w$ is  an ultrametric
              for every strictly positive pseudoultrametrizable weight $w:E(G)\rightarrow \mathbb{R}^+$;
\item [(ii)] The shortest-path pseudometric $d_w$ is a metric
              for every strictly positive pseudometrizable weight $w:E(G)\rightarrow \mathbb{R}^+$.
\end{itemize}
\end{corollary}
Using theorems~\ref{th2.8} and~\ref{th3.3} it is simple to describe the structural properties of  graphs $G$ for which there
 exist  strictly positive pseudoultrametrizable weights $w$ such that $\mathfrak{U}_w$ does not contain any ultrametric.
Recall some definitions.

For a sequence of the sets $A_n$, $n \in \mathbb{N}$,  the \textit{upper limit},
 $\lim\sup_{n\rightarrow \infty} A_n$,
 is the set of elements $a$ such that $a \in A_n$ for infinitely many $n$, i. e.,
$$
\limsup\limits_{n\rightarrow \infty}A_n = \bigcap\limits_{k=1}^{\infty}\left(\bigcup \limits_{n=1}^{\infty}A_{n+k}\right).
$$

The subset $V_0$ of the vertices set of a graph $G$ is called \textit{independent}, if every two vertices from $V_0$ are not adjacent.

\begin{corollary}\label{cor3.2}
Let $G=(V,E)$ be a nonempty connected graph.
The following two statements are equivalent:
\begin{itemize}
\item [(i)] There exists a strictly positive pseudoultrametrizable weight $w$ such that the set $\mathfrak{U}_w$ does not contain any
ultrametric;
\item [(ii)] There exist some vertices $u^*$, $v^* \in V(G)$ and a sequence $\{F_j\}_{j\in \mathbb{N}}$,
 $F_j \in \mathfrak{P}_{u^*,v^*}$, such that
$$
\limsup\limits_{j \rightarrow \infty} V(F_j)
$$
is an independent set.
\end{itemize}
\end{corollary}

Now we shall give some examples of graphs $G$ for which every strictly positive pseudoultrametrizable weight $w:E(G)\to\mathbb{R}^+$
can be extended to an ultrametric.
\begin{example}\label{ex4.7}
If every connected component of nonempty graph $G=(V,E)$ contains at most one vertex of the infinite degree,
then for each strictly positive pseudoultrametrizable weight $w:E(G)\rightarrow \mathbb{R}^+$ there exists
an ultrametric $\rho \in \mathfrak{U}_w$.

Indeed, let $\{G_i:i \in \mathcal{I}\}$ be the set of connected components of $G$.
We complete the graph $G$, if necessary, to the connected graph $\tilde{G}$ as it was done in the proof of Theorem~\ref{th2.4}.
If distinct vertices $u^*$ and $v^*$ lie in the same component  $G_i$, then one of them, for example $u^*$, is incident to
the finite number of edges $e_1, e_2,....,e_n$.
Thus, for an arbitrary infinite sequence of paths $F_i \in \mathfrak{P}_{u^*,v^*}$, one from $e_j$, $j=1,...,n$ belongs to the
 infinite number of paths.

If an arbitrary vertices  $u^*$ and $v^*$ lie in the different connected components, then there exists an edge of the form
$\{v_{i_0},v_i\}$ which belongs to every path  $F_i \in \mathfrak{P}_{u^*,v^*}$.
In both cases, applying Theorem~\ref{th3.3}, we get the existence of an ultrametric $\rho \in \mathfrak{U}_w$.
\end{example}

\begin{example}\label{ex4.8}
If the nonempty graph $G$ is a tree, then the subdominant pseudoultrametric $\rho_w$ is an ultrametric for each strictly
 positive weight  $w:E(G)\rightarrow \mathbb{R}^+$.
Indeed, the well-known characteristic property of trees says that for every two distinct vertices  $u^*, v^* \in V(G)$  there exists
only  one path connected them in $G$.
Consequently, every sequence  $\tilde{F}$ of paths $F_j \in \mathfrak{P}_{u^*,v^*}$ is stationary, $F_1=F_2=...=F_n=F_{n+1}=...$.
This leads to the automatic truth of assertion (ii) from Theorem~\ref{th3.3}.
\end{example}

\section{The least element of $\mathfrak{U}_w$  and uniqueness of pseudoultrametric extension of weight}

In the present section we shall show that only the complete k-partite graphs have the following property: 
the poset $(\mathfrak{U}_w,\preceq)$ contains the least element 
for every pseudoultrametrizable weight $w:E(G)\rightarrow \mathbb{R}^+$.
Moreover,  a uniqueness criterion for the problem of the extension of a weight $w$ to a pseudoultrametric
$\rho:V(G)\times V(G)\rightarrow \mathbb{R}^+$ will be  given also.
The criterion of existence of such extensions was obtained above in Theorem~\ref{th2.4}.

Recall that a graph $G$ is called {\it k-partite}, if the set $V(G)$ can be decomposed into $k$
($k$ is an arbitrary cardinal number)  nonempty disjoint sets $V_{\alpha}$,
$$
V(G)=\bigcup\limits_{\alpha \in I} V_{\alpha},\quad \alpha \in \mathcal{I},\quad |\mathcal{I}|=
k,\quad V_{\alpha_i}\cap V_{\alpha_j}=\varnothing\mbox{ if } i\neq j,
$$
such that for every $\{x,y\}\in E(G)$ the vertices $x$ and $y$ lie in distinct parts $V_{\alpha}$.
 A $k$-partite graph is \textit{complete} if every two vertices in distinct parts are adjacent.
 It is clear that a $k$-partite graph is empty, if $k=1$ and  connected, if  $k \geq 2$.

\begin{figure}[ht] 
\centering
\includegraphics[width=0.35\linewidth]{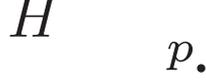}
\caption{If we define a weight $w$ on $E(H)$ such that $w(\{u,v\})>0$, then $(\mathfrak{U}_w,\preceq)$ does not contain the least
 element.}
\label{ris3}
\end{figure}

The proof of the following lemma can be found in  \cite{DDP}.
\begin{lemma}\label{lem5.1}
Let $G$ be a graph with $V(G)\neq \varnothing$.
Then $G$ is a complete k-partite with  $k\geq 1$ if and only if $G$ does not contain any
    induced subgraphs isomorphic to the graph $H$ depicted by Figure~\ref{ris3}.
\end{lemma}

We shall denote by $\mathrm{TM}$ (twice-max) the set of  unordered pairs $p, q$ of distinct
    nonadjacent vertices of the graph $(G,w)$ having the following property:
\textbf{each} path $P\in \mathfrak{P}_{p,q}$ contains at least two distinct edges
$e_1$ and $e_2$ such that $w(e_1)=w(e_2)=\max\limits_{e \in E(P)}w(e)$.

\begin{theorem}\label{th5.2}
The following conditions are equivalent for every nonempty graph $G$:
\begin{itemize}
\item [(i)] The poset $(\mathfrak{U}_w,\preceq)$ contains the least pseudoultrametric $\rho_{0,w}$ for 
            every pseudoultrametrizable weight $w:E(G)\rightarrow \mathbb{R}^+$, i.e., the inequality
                \begin{equation}\label{eq21.1}
                \rho_{0,w}(u,v)\leq \rho(u,v)
                \end{equation}
            holds for every $\rho \in \mathfrak{U}_w$ and all $u,v \in V(G)$;
\item [(ii)] $G$  is a complete k-partite graph with $k \geq 2$.
\end{itemize}
If condition (ii) holds and $w$ is  a pseudoultrametrizable weight, then for $u\neq v$ we have
\begin{equation}\label{eq21.2}
\rho_{0,w}(u,v)=\left\{
\begin{array}{l}
0 \mbox{ if   } \{u,v\} \in \mathrm{TM}  \\\
\max\limits_{e\in E(F)}w(e) \mbox{ if } \{u,v\} \notin \mathrm{TM} \\
\end{array}
\right.
\end{equation}
where $F$ is an arbitrary path from $\mathfrak{P}_{u,v}$ for which $\max\limits_{e\in E(F)}w(e)$ is achieved on a single edge.
\end{theorem}
\begin{proof}
$\textbf{(i)}\Rightarrow\textbf{(ii)}$  Suppose (ii) does not hold.
Then, by Lemma~\ref{lem5.1}, there exist vertices $u,v$, $\{u,v\}\in E(G)$, and a vertex $p \in V(G)$, $u\neq p\neq v$
such that  $\{u,p\}\notin E(G)$ and  $\{v,p\}\notin E(G)$.
Define the weight $w(e)=1$ for every $e \in E(G)$.
Consider the following two pseudoultrametrics on the set $V(G)$:

$\rho_1(u,p)=\rho_1(p,u)=\rho_1(s,s) =0$  for all $s\in V(G)$ and $\rho_1(s,t)=1$ in the opposite case;

$\rho_2(v,p)=\rho_2(p,v)=\rho_2(s,s) =0$ for all $s\in V(G)$ and $\rho_2(s,t)=1$ in the opposite case.

It is clear that $\rho_1,\rho_2 \in \mathfrak{U}_w$.
Assuming that there exists the least pseudoultrametric $\rho \in \mathfrak{U}_w$,  we get the contradiction
$$
1=\rho(u,v)\leq \rho(u,p)+\rho(p,v)\leq (\rho_{1} \wedge \rho_{2})(u,p)+(\rho_{1} \wedge \rho_{2})(p,v)=0.
$$
The implication (i)$\Rightarrow$(ii) follows.

$\textbf{(ii)}\Rightarrow\textbf{(i)}$  Let condition (ii) hold and $(G,w)$ be an arbitrary
weighted graph with a pseudoultrametrizable $w$.
Since $k \geq 2$,  $G$ is a connected graph, as it was mentioned above.
Let us show that $\rho_{0,w}$ defined by~(\ref{eq21.2}) is the least element of the poset $\mathfrak{U}_w$.

Note that the function $\rho_{0,w}$ is well defined.
Indeed, suppose that there exist $\{u,v\}\notin \mathrm{TM}$ and two distinct paths
$F_1, F_2 \in \mathfrak{P}_{u,v}$ such that each of them  contains only one edge of maximal weight.
Let $\rho \in \mathfrak{U}_w$.
Consider in the pseudoultrametric space $(V(G),\rho)$ the cycle generated by the path $F_1$ and the edge $\{u,v\}$.
By Lemma~\ref{lem2.2} we obtain that $\rho(u,v)=\max\limits_{e\in E(F_1)}w(e)$.
Similarly if we consider the cycle generated by $F_2$ and $\{u,v\}$, then 
$\rho(u,v)= \max\limits_{e\in E(F_2)}w(e)$.

At the same time for every edge $\{u,v\} \in E(G)$ we have $\rho_{0,w}(u,v)=w(\{u,v\})$,
since in this case $\{u,v\}\notin \mathrm{TM}$ and the path $(u,v)$ is one from the paths connecting the vertices $u$ and $v$.

Let us prove that the function $\rho_{0,w}$  is  actually a pseudoultrametric.
It is sufficient to establish the strong triangle inequality
\begin{equation}\label{eq5.3*}
\rho_{0,w}(x,y)\leq \rho_{0,w}(x,z)\vee \rho_{0,w}(z,y)
\end{equation}
for pairwise distinct vertices $x,y,z \in V(G)$.

If all three points $x, y, z$ are pairwise adjacent, then~(\ref{eq5.3*}) follows from the pseudoultrametrizability of $w$.
Let us show that~(\ref{eq5.3*}) holds if among the vertices there are only two adjacent pairs.
If the weights of corresponding edges are distinct, then,
according to~(\ref{eq21.2}), the weight of the missing edge will be equal to the maximum weight of those edges.
The inequality~(\ref{eq5.3*}) holds again.
Assume that only two pairs of vertices are adjacent  and $w(\{x,z\})=w(\{z,y\})=\rho_{0,w}(x,z)=\rho_{0,w}(z,y)=a$
but $\rho_{0,w}(x,y)=b>a$.
Then there exists a path $F\in \mathfrak{P}_{x,y}$ with a unique edge $e_0$ having the maximal weight.
Consider the cycle consisting of the path $F$ and the edges $\{x,z\}$ and $\{z,y\}$
(or one of these edges, depending on whether the path $F$ contains one of them).
This cycle contains the unique edge $e_0$ of maximal weight, contrary to pseudoultrametrizability of $w$
(see Theorem~\ref{th2.4}).

By Lemma~\ref{lem5.1} the case when exactly two vertices are adjacent is impossible.
Therefore we may suppose that the vertices $x, y, z$ are pairwise non-adjacent.
Let $\rho_{0,w}(x,y)=a>0$ (if $\rho_{0,w}(x,y)=0$, then inequality~(\ref{eq5.3*}) is evident).
Since $a>0$, there exists a path $F \in \mathfrak{P}_{x,y}$ with the unique edge $e_0=\{u_0,v_0\}$ of maximal weight $w(e_0)=a$.
Suppose first that the path $F$ does not pass through the point $z$.
From (ii) it follows that at least one from the points $u_0, v_0$, for example $u_0$, is adjacent with $z$.
Consider the following two paths: $F_1 \in \mathfrak{P}_{x,z}$ consisting of the edge  $\{u_0,z\}$ and some part of the path $F$
and the path $F_2 \in \mathfrak{P}_{y,z}$ also consisting of the edge $\{u_0,z\}$ and the rest of $F$.
We may assume, without  loss of generality, that $\{u_0,v_0\}\in E(F_2)$.
If $w(\{u_0,z\})>a$, then $\{u_0,z\}$ is the unique edge of maximal weight for both $F_1, F_2$.
Hence, from~(\ref{eq21.2}) we have
$$
\rho_{0,w}(x,z)=w(\{u_0,z\})=\rho_{0,w}(z,y)>a,
$$
so that,~(\ref{eq5.3*}) follows.
Let $w(\{u_0,z_0\})\leqslant a$.
If the last inequality is strict, then $\{u_0,v_0\}$ is the unique edge of maximal weight in $F_2$. Consequently
$$
\rho_{0,w}(y,z)=w(\{u_0,v_0\})=a,
$$
so that~(\ref{eq5.3*}) holds.
If $w(\{u_0,z_0\})=a$, then $\{u_0,z_0\}$ is the unique edge of maximal weight in $F_1$ and again we obtain~(\ref{eq5.3*}).
Consider the case when the path $F$ passes through the point $z$.  By  splitting $F$ into the two paths $F_1\in \mathfrak{P}_{x,z}$
and $F_2\in \mathfrak{P}_{z,y}$, we obtain that one of the values $\rho_{0,w}(x,z)$, $\rho_{0,w}(z,y)$ is equal to $a$.
Thus,~(\ref{eq5.3*}) follows again.

It remains to prove that inequality~(\ref{eq21.1}) holds for every  $\rho \in \mathfrak{U}_w$ and $u,v \in V(G)$.
This is trivial, if $\rho_{0,w}(u,v)=0$.
Assume that $\rho_{0,w}(u,v)=a>0$, then $\{u,v\}\notin \mathrm{TM}$ and there exists a path $F \in \mathfrak{P}_{u,v}$ with
the unique edge $e_0$ of maximal weight $w(e_0)=a$.
Suppose that $\rho(u,v)<a$.
Then the function $\rho(x,y)$ is not a pseudoultrametric because 
the cycle $C$ with $V(C)=V(P)$ and $E(C)=E(P)\cup\{\{u,v\}\}$ contains
exactly one edge of maximal weight.
\end{proof}

As an application of Theorem~\ref{th5.2} we shall obtain a characterization of the stars.
Recall that a {\it star} is a complete bipartite graph for  which at least one of the parts is a singleton.
\begin{corollary}\label{cor6*}
The following conditions are equivalent:
\begin{itemize}
\item [(i)] Every weight $w:E(G)\rightarrow \mathbb{R}^+$ is pseudoultrametrizable and the poset
            $(\mathfrak{U}_w,\preceq)$ contains the least element;
\item [(ii)] $G$ is a star.
\end{itemize}
\end{corollary}
\begin{proof}
The implication (ii)~$\Rightarrow$~(i) follows from Theorem~\ref{th5.2} and Corollary~\ref{cor2.11}.
Let (i) hold, then again using Theorem~\ref{th5.2} and Corollary~\ref{cor2.11} we obtain
that $G$  is a complete k-partite graph with $k \geq 2$ being at the same time a tree.
If $k \geq 3$, then $G$ contains a triangle. To construct it, we may take
three points lying in three distinct parts of the graph $G$.
Since the trees do not contain  cycles, $k=2$.
If each part of $G$ contains at least two points, then it is easy to construct a
 quadruple (4-cycle) $C \subseteq G$, contrary again to the acyclic property of $G$.
Thus, $G$ is a complete  bipartite graph one part of which is a singleton.
\end{proof}

We turn now to the conditions of uniqueness in the problem of extension of a weight $w$ by pseudoultrametrics.
We need the following
\begin{definition}\label{def5.4}
Let $(G,w)$ be a nonempty weighted graph and let $u$,  $v$ be two distinct disjoint vertices of $G$.
We shall say that $u$ and $v$ are \textit{well chained} if for every $\varepsilon >0$ there exists a path $u=u_1,u_2,....,u_n=v$ such that
$\{u_i,u_{i+1}\}\in E(G)$ and $w(\{u_i,u_{i+1}\})\leq \varepsilon$ for $i=1,...,n-1$.
\end{definition}
We denote the set of all such pairs $\{u, v\}$ by $\mathrm{WCh}$.

\begin{remark}\label{rem5.5}
 The notion ``well chained points'' is often  used in the metric continuum theory \cite[с. 60]{SBNJ} and it plays an important
 role in the considering of problems related to the connectivity in metric spaces (see, e.g., \cite{AGOF}).
\end{remark}
\begin{remark}\label{rem5.55}
The notion ``well chained points'' arises naturally in the study of subdominant ultrametrics (see \cite{RTV} and \cite{Le}).
In particular, it is easy to show that, for a strictly positive pseudoultrametrizable weight $w$ and a connected graph $G$,
some vertices $u$ and $v$, $u\neq v$ are well chained if and only if $\rho_w(u,v)=0$.
\end{remark}

\begin{theorem}\label{th5.6}
Let $(G,w)$ be a nonempty connected weighted graph with a pseudoultrametrizable weight $w$.
The set $\mathfrak{U}_w$ contains only one element if and only if
\begin{equation}\label{eq5.4}
\mathrm{TM}  \subseteq \mathrm{WCh}.
\end{equation}
\end{theorem}
\begin{proof}
Let~\eqref{eq5.4} hold. In order to  prove the uniqueness of extension of $w$ it is sufficient to establish
the equality
\begin{equation}\label{eq5.5}
\rho(u,v)=\rho_w(u,v)
\end{equation}
for every pair of distinct nonadjacent vertices $u$, $v$ and every $\rho \in \mathfrak{U}_w$.
Let $u$, $v$ be distinct nonadjacent vertices for which $\{u,v\} \notin \mathrm{TM} $.
In this case, arguing as in the verification of correctness of the definition ~\eqref{eq21.2},
we see that the value $\rho(u,v)$ does not depend on the choice of $\rho \in \mathfrak{U}_w$.
Hence, ~\eqref{eq5.5} holds.
Consider now the case  when $\{u,v\}\in \mathrm{TM}$.
In view of inclusion~\eqref{eq5.4} the vertices are well chained.
This and the definition of the subdominant pseudoultrametric $\rho_w$ imply
\begin{equation}\label{eq5.6}
\rho_w(u,v)=0.
\end{equation}
Since $G$ is a connected graph, by Theorem~\ref{th2.8} we obtain $\rho(u,v) \leq \rho_w(u,v)$.
Moreover, $0\leq \rho(u,v)$ for $u,v \in V(G)$.
The past two inequalities and~\eqref{eq5.6} imply~\eqref{eq5.5}.

Let~\eqref{eq5.4} be false.
Let us prove the nonuniqueness of the extensions of $w$.

Let $\{u_0, v_0\}\in \mathrm{TM} \backslash \mathrm{WCh}$.
We have $\rho_w(u_0,v_0)>0$ because $\{u_0,v_0\}\notin \mathrm{WCh}$,  so that there exists $\varepsilon_0>0$ for which
\begin{equation}\label{eq5.7}
\max\limits_{e \in E(P)}w(e)\geq \varepsilon_0
\end{equation}
for every $P \in \mathfrak{P}_{u,v}$.
Let $\varepsilon_1$  be $\varepsilon_2$ two distinct numbers  from $[0,\varepsilon_0)$.
Consider the graph $\tilde{G}$ with
$$
V(\tilde{G})=V(G) \mbox{ и } E(\tilde{G})=E(G)\cup \{\{u_0,v_0\}\},
$$
i.e., $\tilde{G}$ can be obtained from $G$ by adding the edge $\{u_0,v_0\}$.
We define on $\tilde{G}$ the weights $w_i:E(\tilde{G})\rightarrow \mathbb{R}^+$, $i=1,2$ by the rule
\begin{equation*}
w_i(e)=\left\{
\begin{array}{l}
w(e) \mbox{ if } e \in E(G)\\\
\varepsilon_i \mbox{ if } e=\{u_0,v_0\}.\\
\end{array}
\right.
\end{equation*}
Let us verify that the weight $w_1$ is pseudoultrametrizable.
In accordance with Theorem~\ref{th2.4} it is sufficient to show that for every cycle
 $C\subseteq \tilde{G}$ there are  two distinct edges $e_1$, $e_2 \in E(C)$ satisfying
\begin{equation}\label{eq5.8}
\max\limits_{e \in E(C)}w_1(e)=w_1(e_1)=w_1(e_2).
\end{equation}
Let $C\subseteq \tilde{G}$. If $\{u_0,v_0\} \notin E(C)$, then the existence of such $e_1$, $e_2$ follows from the pseudoultrametrizability of $w$.
Suppose that $\{u_0,v_0\} \in E(C)$.
We can get a path $P \in \mathfrak{P}_{u_0,v_0}$ by removing the edge $\{u_0,v_0\}$ from the cycle $C$.
Since $\{u_0,v_0\} \in \mathrm{TM} $, there are two distinct edges $e_1, e_2 \in E(P)$ such that
$$
\max\limits_{e \in E(P)}w_1(e)=\max\limits_{e \in E(P)}w(e)=w(e_1)=w(e_2)=w_1(e_1)=w_1(e_2).
$$
In virtue of inequalities~\eqref{eq5.7} and $\varepsilon_1<\varepsilon_0$ we see that~\eqref{eq5.8} holds for these edges.
Consequently, the weight $w_1:E(\tilde{G})\rightarrow \mathbb{R}^+$ is pseudoultrametrizable.
In the same way we can verify that $w_2:E(\tilde{G})\rightarrow \mathbb{R}^+$ is also a pseudoultrametrizable weight.
Let $\rho_1$ and $\rho_2$ be  two pseudoultrametrics on $V(\tilde{G})=V(G)$ extending $w_1$ and, respectively, $w_2$.
Then it is clear that $\rho_1$ and $\rho_2$ extend also the weight $w$.
Since $\varepsilon_1 \neq \varepsilon_2$, we have $w_1 \neq w_2$. Thus $\rho_1 \neq \rho_2$.
Consequently the extension of the  weight $w$ to pseudoultrametrics  is not unique if  condition~\eqref{eq5.4} is violated.
\end{proof}

\begin{remark}\label{rem5.7}
If $G$ is a nonempty disconnected graph, then, for every pseudoultrametrizable weight $w$, the inequality
$$
\mathrm{card} (\mathfrak{U}_w)\geq \mathfrak{c}
$$
holds where $\mathfrak{c}$, as usual, is the cardinality of continuum.
The last inequality immediately follows from~\eqref{eq2.8}.
\end{remark}
\begin{example}\label{ex5.8}
For complete k-partite graphs with $k \geq 2$ the uniqueness of extension of $w$ is equivalent to the equality
\begin{equation}\label{eq5.9}
\rho_w=\rho_{0,w},
\end{equation}
where $\rho_w$ is the subdominant  pseudoultrametric, and $\rho_{0,w}$ is  defined by equality~\eqref{eq21.2}.
Indeed, every complete k-partite graph is connected for $k\geqslant 2$.
In correspondence with theorems~\ref{th2.4}  and~\ref{th5.2} the pseudoultrametric $\rho_w$ is the greatest element of
$(\mathfrak{U}_w,\preceq)$ and $\rho_{0,w}$ is the least element of this poset.
Consequently, for every $\rho \in \mathfrak{U}_w$, we have
$$
\rho_{0,w}\preceq \rho \preceq \rho_w,
$$
and by~(\ref{eq5.9})
$$
\rho_{0,w}= \rho = \rho_w.
$$
Let us show also that equality~\eqref{eq5.9} is equivalent to inclusion~(\ref{eq5.4}).
Note that in the proof of Theorem~\ref{th5.2} the equality
\begin{equation}\label{eq5.10}
\rho_{0,w}(u,v)=\rho_{w}(u,v)
\end{equation}
was established for $\rho_{0,w}(u,v)>0$.
Moreover, \eqref{eq5.10} must be hold for adjacent $u$ and $v$ because $\rho_{0,w},\rho_{w}\in \mathfrak{U}_w$.
Directly from the definitions we have
\begin{equation}\label{eq5.11}
\mathrm{WCh} = \{\{u,v\}:u  \mbox{ and } v \mbox{ are non-adjacent, } u \neq v \mbox{ and } \rho_w(u,v)=0\},
\end{equation}
and, for complete k-partite graphs with $k\geq 2$,
\begin{equation*}
\mathrm{TM} = \{\{u,v\}:u  \mbox{ and } v \mbox{ are non-adjacent, } u \neq v \mbox{ and } \rho_{0,w}(u,v)=0\}.
\end{equation*}
Thus,~\eqref{eq5.9} is equivalent to the equality $\mathrm{WCh} =\mathrm{TM} $.
By $\rho_{0,w}\preceq \rho_w$, we have the inclusion $\mathrm{TM}  \supseteq\mathrm{WCh} $.
Consequently ~\eqref{eq5.9} is equivalent to the converse inclusion $\mathrm{TM}  \subseteq \mathrm{WCh} $.
\end{example}
\begin{example}\label{ex5.9}
Let $G$ be a tree and $w:E(G)\rightarrow \mathbb{R}^+$ be a strictly positive weight.
The weight $w$ has the unique extension if and only if $\mathrm{TM}  = \varnothing$.

Indeed, every two vertices of the tree $G$ are connected by the unique path.
From this uniqueness, the strict positiveness of $w$, formula~\eqref{eq5.11} and the definition of  $\rho_w$  it follows
that $\mathrm{WCh}  =\varnothing$.
By Theorem~\ref{th5.6}, the last equality implies that the uniqueness of extension of $w$  is equivalent to
 $\mathrm{TM}=\varnothing$.
\end{example}

\begin{remark}\label{rem5.10}
The equality $\mathrm{TM}  =\varnothing$ is equivalent to the fact that the problem of extending of strictly positive
pseudoultrametrizable weight $w$ to a pseudoultrametric has the unique solution for every $G$ satisfying condition (ii) of Theorem~\ref{th3.3}.
For trees this equality is equivalent to the following statement.

If $e_1$ and $e_2$ are two distinct edges of the tree $G$ such  that $w(e_1)=w(e_2)$, then for every path $P\subseteq G$
including $e_1$ and $e_2$ there exists an edge $e_3 \in E(P)$ such that $w(e_3)>w(e_1)$.
\end{remark}

\newpage

{\bf Oleksiy Dovgoshey}

Institute of Applied Mathematics and Mechanics of NASU, R. Luxemburg str. 74, Donetsk 83114, Ukraine

{\bf E-mail: } aleksdov@mail.ru
\bigskip

{\bf Evgeniy Petrov}

Institute of Applied Mathematics and Mechanics of NASU, R. Luxemburg str. 74, Donetsk 83114, Ukraine

{\bf E-mail: } eugeniy.petrov@gmail.com

\end{document}